\documentclass[11pt]{amsart}

\usepackage{enumerate}
\usepackage{graphicx}
\usepackage{amsthm,amscd,amssymb}
\usepackage{amsfonts}
\usepackage{color}
\usepackage{srcltx} %jump between tex and dvi, used for kile
\newtheorem*{thm*}{Theorem I}
\numberwithin{equation}{section}
\newtheorem{thm}{Theorem}
\numberwithin{thm}{section}
\newtheorem{lemma}[thm]{Lemma}

\newtheorem{cor}[thm]{Corollary}

\theoremstyle{definition}
\newtheorem{definition}[thm]{Definition}

\theoremstyle{remark}

\newcommand{\ds}{\displaystyle}
\newcommand{\R}{\mathbb{R}}
\DeclareMathOperator*{\esssup}{ess\,sup}

\newcommand{\norm}[1]{\left\| #1\right\|}
\newcommand{\Lp}[2]{L^{#1}\left(#2\right)}
\newcommand{\dive}[1]{\;{\rm div}\; #1} 
\newcommand{\Ri}[1]{{\rm {\bf I}_{#1}}}

\newcommand{\curl}{\;{\rm curl}\;} 
\newcommand{\M}{\mathcal{M}}
\begin{document}

\pagestyle{headings}

\title[Leray's self-similar solutions in Marcinkiewicz spaces]{Leray's self-similar solutions to the Navier--Stokes equations with profiles in  Marcinkiewicz and Morrey spaces}

\author[Cristi Guevara]{Cristi Guevara}
\address{Department of Mathematics,
Louisiana State University,
303 Lockett Hall, Baton Rouge, LA 70803, USA.}
\email{cguevara@lsu.edu}

\author[Nguyen Cong Phuc]
{Nguyen Cong Phuc$^*$}
\address{Department of Mathematics,
Louisiana State University,
303 Lockett Hall, Baton Rouge, LA 70803, USA.}
\email{pcnguyen@math.lsu.edu}

\thanks{$^*$Supported in part by the Simons Foundation, Award Number: 426071}

%\today

\begin{abstract}
We rule out the existence of Leray's backward self-similar solutions to the Navier--Stokes equations with profiles in $L^{12/5}(\R^3)$ or in the Marcinkiewicz space 
$L^{q,\infty}(\R^3)$ for $q\in(12/5, 6)$. This follows from a more general result formulated in terms of Morrey spaces and the first order Riesz's potential.
\end{abstract}

\maketitle

\section{Introduction}

The motion of an incompressible fluid in three spatial dimensions, $\R^3,$  with viscosity $\nu>0$  and zero external force is described 
by the Navier--Stokes equations 
\begin{eqnarray}\label{NSE}
\left\{\begin{array}{c}
u_t-\nu\Delta u+u\cdot\nabla u+\nabla p=0,  
\\ \dive u=0,
\end{array}\right.
\end{eqnarray}
with an initial condition $ u(x, 0) = u_0(x).$  Here  the unknown velocity  $u = u(x,t) = (u_1(x,t),u_2(x,t),u_3 (x,t))\in  \R^3$ and the unknown  pressure $p = p(x,t)\in \R$ are defined for each position $x \in \R^3$ and time $t\geq0$.

Since Leray's work \cite{Leray34} in 1934, there has been a long-standing question whether  solutions to \eqref{NSE} develop a singularity in finite time, or whether \eqref{NSE} admits a time-global smooth solution for any given smooth and compactly supported initial datum $u_0$. To look for a singular solution, Leray \cite{Leray34} suggested to consider the backward self-similar solutions of \eqref{NSE}, i.e.,  those  of the form
\begin{equation}\label{LSln}
u(x,t)=\lambda(t)U(\lambda(t)x), \qquad  p(x,t)=\lambda^2(t)P(\lambda(t)x),
\end{equation}
where 
$$\lambda(t)=\dfrac{1}{\sqrt{2a(T-t)}}, \quad a>0, \quad  T>0,$$
and 
$U=(U_1,U_2,U_3)$ and $P$ are  defined in the whole $\R^3$. Note then that $u$ is defined in $\R^3\times (-\infty, T)$ and  if the profile $U$ is not identically zero then $u$ given by \eqref{LSln} develops a singularity at time $t=T$. Here  certain natural energy norms of $u$ should be required to be finite. For otherwise
the profile $U=\nabla \Phi$ and $P=-\frac{1}{2}|U|^2 -a y\cdot U$, for any non-zero harmonic function $\Phi$, would immediately yield a non-trivial self-similar solution.

By direct calculations,  one finds that $(u, p)$ of the form \eqref{LSln} is a solution of  \eqref{NSE} if and only if
 $(U, P)$  solves the following nonlinear time-independent system in $\R^3$:
\begin{eqnarray}\label{SNSE}
\left\{\begin{array}{c}
-\nu\Delta U+aU+a(y\cdot \nabla) U+(U\cdot \nabla)U+ \nabla P=0,  
\\ \dive U=0.
\end{array}\right.
\end{eqnarray}

Leray's question  was open  until 1996 when  Ne{\v{c}}as, R{\r u}{\v{z}}i{\v{c}}ka,  and  {\v S}ver\'ak \cite{NeRuSv96} showed that there does not exist a non-trivial solution of the form \eqref{LSln} with finite kinetic energy and satisfies the natural {\it global} energy inequality
\begin{equation}\label{global-energy}
 \int_{\R^3}\dfrac12|u(x, t)|^2dx+\int_{t_1}^{t}\int_{\R^3}\nu|\nabla u(x,t)|^2dxdt\leq\int_{\R^3}\dfrac12| u(x, t_1)|^2dx
\end{equation}
for all $t\in (t_1, t_2)$. More generally, they proved that if $U\in W^{1,2}_{\rm loc}(\R^3)\cap L^3(\R^3)$ is a weak solution of \eqref{SNSE} then $U\equiv 0$.
Note that the  $L^3$   {\it global} integrability condition of $U$ holds if the corresponding self-similar solution $u$   satisfies the global energy estimates \eqref{global-energy}.

On the other hand, \cite{NeRuSv96} left open the question of  existence of self-similar singularities which satisfy only the {\it local} energy inequality 
\begin{equation}\label{local-energy}
\esssup_{t_3<t<T} \int_{B_r(x_0)}\dfrac12|u(x, t)|^2dx+\int_{t_3}^{T}\int_{B_r(x_0)}\nu|\nabla u(x,t)|^2dxdt<+\infty,
\end{equation}
for some ball $B_r(x_0)$ and some $t_3<T$. This question was later answered by Tsai in \cite{Tsai98}, where he showed that backward self-similar solutions to \eqref{NSE} satisfying  \eqref{local-energy} must also be zero:

\begin{thm}[Tsai \cite{Tsai98}]\label{T1.1} Suppose $u$ is a weak solution of \eqref{NSE} satisfying the finite local energy condition
\eqref{local-energy} for some ball $B_r(x_0)$ and some $t_3<T$. If $u$ is of the form
\eqref{LSln}, then $u\equiv 0$.
\end{thm}

In the same paper Tsai also extended the result of \cite{NeRuSv96} to a super critical range of the integrability condition on the profile $U$:

\begin{thm}[Tsai \cite{Tsai98}]\label{T1.2}
If a weak solution $U\in W^{1,2}_{\rm loc}(\R^3)$ of \eqref{SNSE} belongs $L^{q}(\R^3)$ for some $q\in (3,\infty]$ 
then it must be zero provided $q\not=\infty$ and constant provided $q=\infty$.
\end{thm}

The proofs of Theorems \ref{T1.1} and  \ref{T1.2} use the important fact that the scalar function 
$$\Pi(y)=\frac 12 |U(y)|^2+P(y)+a y\cdot U(y)$$
satisfies the maximum principle. 
%as it satisfies the differential inequality $$-\nu \Delta \Pi + (U +a y)\cdot \nabla \Pi =- |{\rm curl~} U|^2 \leq 0.$$
This idea had been employed earlier in the work of  Ne{\v{c}}as, R{\r u}{\v{z}}i{\v{c}}ka,  and  {\v S}ver\'ak \cite{NeRuSv96} to treat the critical case 
$q=3$ mentioned above. See also the earlier work \cite{FR94A, FR94B, Stru95} in the context of stationary Navier--Stokes equations in higher dimensions.

We mention that the proof of Theorem \ref{T1.1} in \cite{Tsai98}  also makes use of a celebrated $\epsilon$-regularity criterion due to  Caffarelli, 
Kohn, and Nirenberg \cite{CKN82} to show that if $(u,p)$ of the form \eqref{LSln} is a {\it suitable weak solution} of \eqref{NSE} in 
$B_1(0)\times (T-1,T)$ then $U(y)=O(|y|^{-1})$ as $y\rightarrow\infty$ (see \cite[Corollary 4.3]{Tsai98}).

On the other hand, one important step in Tsai's proof of Theorem \ref{T1.2} is to show that if $U$ is a weak solution of \eqref{SNSE} and $U\in L^q(\R^3)$ for some $3<q<\infty$
then  $U=o(|y|)$ as $y\rightarrow\infty$ (see \cite[Lemma 3.3]{Tsai98}). He also remarked that his approach to this pointwise asymptotic estimate 
fails at the end-point case $q=3$ and suggested that the sub-critical case $q<3$ would require a different idea (see \cite[Remark 3.2]{Tsai98}).

A main goal of this paper is to improve the result of \cite{NeRuSv96} by allowing the profile $U$ to be in  spaces strictly larger than $L^3(\R^3)$. Furthermore,  by that way we also extend the result of Theorem \ref{T1.2} to the sub-critical range $q\in [\frac{12}{5}, 3)$. Indeed, we prove
\begin{thm}\label{weakLq}
Let $U\in W^{1,2}_{\rm loc}(\R^3)$ be a weak solution  of \eqref{SNSE}. If $U\in L^{q, \infty}(\R^3)$ for some $q\in(\frac{12}{5},6)$ or if $U\in L^{\frac{12}{5}}(\R^3)$ then it must be identically zero. 
\end{thm}

In the above theorem, the space  $L^{q, \infty}(\R^3)$ is a the Marcinkiewicz space (or weak $L^q$ space) defined as 
 the set of measurable functions $g$ in $\R^3$ such that the quasinorm 
 $$\|g\|_{L^{q, \infty}(\R^3)}:=\sup_{\alpha >0} \,\alpha \, |\{x\in \R^3: |g(x)|>\alpha\}|^{\frac{1}{q}}<+\infty.$$ 

It is well-known that $L^q(\R^3)\subset L^{q, \infty}(\R^3)$, or more generally $L^{q, s}(\R^3) \subset L^{q, \infty}(\R^3)$ for any $s> 0$, where 
$L^{q, s}(\R^3)$, $q>0, s\in (0,\infty)$, is the Lorentz space with quasinorm
$$\|g\|_{L^{q, s}(\R^3)}:=\left(q \int_{0}^{\infty}\alpha^{s} |\{x\in\R^3: |g(x)|>\alpha\}|^{\frac{s}{q}} \frac{d\alpha}{\alpha}
\right)^{\frac{1}{s}}.$$ 
Note that $L^{q}(\R^3)=L^{q, q}(\R^3)$ and $L^{q, s_1}(\R^3) \subset L^{q, s_2}(\R^3)$ for $0<s_1\leq s_2$.

Recently in Phuc \cite{Phuc14}, it is shown that locally finite energy solutions to the Navier--Stokes equation \eqref{NSE}  belonging to $L^\infty_{t}(L_{x}^{3,s})$ 
are regular provided $s\not=\infty$. That result strengthens the above mentioned result of \cite{NeRuSv96} as it rules out the existence of self-similar solutions of the form \eqref{LSln} with profiles  $U\in L^{3, s}(\R^3)$ provided $s\not=\infty$.   Thus, in the case $q=3$ Theorem \ref{weakLq} provides the answer to the end-point case $s=\infty$. Note on the other hand that it is still unknown  whether  $L^\infty_{t}(L_{x}^{3,\infty})$ solutions to the Navier--Stokes equations are regular. %This is a difficult end-point case that has been excluded in the study \cite{Phuc14}. 

In fact, we shall prove a more general result than Theorem \ref{weakLq} which allows the profile $U$ to have a very modest decay at infinity. To describe it, recall that the Riesz potential ${\bf I}_{\alpha} $, $\alpha\in (0,3)$, on $\R^3$  is defined by 
$${\bf I}_{\alpha} f(x)=c(\alpha)\int_{\R^3}\dfrac{f(y)}{|x-y|^{3-\alpha}} dy, \qquad x\in\R^3,$$
for  $f\in L^1_{\rm loc}(\R^3)$ such that $\int_{|x|\geq 1} |x|^{\alpha-3}||f(x)| dx<+\infty$. Here the normalizing constant $$c(\alpha)=\frac{\Gamma(\frac{3}{2}-\frac{\alpha}{2})}{\pi^{3/2}2^\alpha\Gamma(\alpha/2)}.$$

Additionally, we define the Morrey space $\M^{p,\gamma}(\R^3)$, $p\geq 1$, $0<\gamma\leq 3$, to be the set of all functions $f\in L_{\rm loc}^p(\R^3)$ such that
$$\int_{B_r(x)}|f(y)|^pdy\leq C\, r^{3-\gamma},$$
for all $x \in \R^3$ and $r>0$ with a constant $C$ independent of $x$ and $r$. The norm $\norm{f}_{\M^{p,\gamma}(\R^3)}$ is given by
$$\norm{f}_{\M^{p,\gamma}(\R^3)}:=\sup_{x\in\R^3,\; r>0}r^{\frac{\gamma-3}{p}}\norm{f}_{\Lp{p}{B_r(x)}}.$$
Obviously, when $\gamma=3$ we have $\M^{p, \gamma}(\R^3)=L^p(\R^3)$. The interest of using such a notation for Morrey spaces is to emphasize that the second index $\gamma$ acts like the dimension in the Sobolev type embedding theorem. Indeed, it is now well-known from the work of D. R. Adams \cite{Adams75} that ${\bf I}_{\alpha}$ continuously maps  $\M^{p,\gamma}(\R^3)$ into $\M^{\frac{\gamma p}{\gamma-\alpha p}, \gamma}(\R^3)$ provided $1<p< \gamma/\alpha$. Thus when $\gamma=3$  the classical Sobolev  embedding theorem 
is recovered.

We are now ready to state the next result of the paper:
\begin{thm}\label{Morrey-loc}
Let  $U\in W^{1,2}_{\rm loc}(\R^3)$ be  a weak solution of \eqref{SNSE}. If for some  $\gamma\in(0,3]$ it holds that
\begin{equation}\label{local-cond-on-U}
\int_{B_r}{\bf I}_1(\chi_{B_r}|U|^2)^2 dx\leq C\, r^{3-\gamma} \qquad \forall {\rm ~balls~} B_r\subset\R^3,
\end{equation}
 where $\chi_{B_r}$ is the characteristic function of $B_r$, then $U\equiv 0$. In particular, if
${\bf I}_1(|U|^2)\in \M^{2,\gamma}(\R^3)$ for some $\gamma\in(0,3]$ then $U$ must be identically zero. 
\end{thm}

We now show that Theorem \ref{Morrey-loc} actually implies Theorem \ref{weakLq}. Indeed,  if $U\in L^{12/5}(\R^3)$ then the Sobolev embedding theorem implies that ${\bf I}_1(|U|^2)\in L^{2}(\R^3)=\M^{2,3}(\R^3)$. Thus by Theorem \ref{Morrey-loc}, 
 if $U$ is also a weak solution of \eqref{SNSE} then $U\equiv 0$. Similarly, if $U\in L^{q,\infty}(\R^3)$ with $q\in(\frac{12}{5}, 6)$ then 
 ${\bf I}_1(|U|^2)\in L^{\frac{3q}{6-q},\infty}(\R^3)$. Since $q>12/5$ it follows that $\frac{3q}{6-q}>2$ and thus  
 by H\"older's inequality we have 
$$\int_{B_r(x)} {\bf I}_1(|U|^{2})^2 dy \leq C \| {\bf I}_1(|U|^{2})\|^{2}_{L^{3q/(6-q),\infty}(\R^3)}\, r^{3- \frac{2(6-q)}{q}}.$$
 This yields ${\bf I}_1(|U|^{2})\in \M^{2, \gamma}(\R^3)$ with $\gamma=\frac{2(6-q)}{q}\in (0, 3)$,  and thus Theorem \ref{weakLq} is a consequence of Theorem \ref{Morrey-loc}.

On the other hand, using Adams Embedding Theorem \cite{Adams75} we have another corollary of Theorem \ref{Morrey-loc}.

\begin{cor}
If  $U\in W^{1,2}_{\rm loc}(\R^3)$ is  a weak solution of \eqref{SNSE} such that  if $U\in \M^{\frac{4\gamma}{2+\gamma}, \gamma}(\R^3)$ for some $\gamma\in(2,3]$ (or  equivalently if $U\in \M^{p, \frac{2p}{4-p}}(\R^3)$ for some $p\in (2, 12/5]$), then $U\equiv0$. 
\end{cor}

The proof of Theorem \ref{Morrey-loc} will be given in Section \ref{sec4}. Surprisingly, it is based on an application of Theorem \ref{T1.1} above. For that a pressure profile $P$ is built from $U$  so that the norm of $P$ in a Sobolev space of {\it negative} order (localized in each ball) is well controlled. Here one has to treat $P$ as a {\it signed} distribution 
in $\R^3$ as no control of $|P|$ is available. This suggests a natural way to control the nonlinear and the pressure
terms in the {\it energy equality}, and a sort of  bootstrapping argument based on the energy equality eventually completes the proof.

\section{Preliminaries}
Throughout the paper we denote by  $B_r(x)$ the open ball centered at $x\in\R^3$ with radius $r>0$, i.e.,
$$B_r(x)=\{y\in \R^3:|x-y|<r\}.$$
%and
%$$Q_r(z)=B_r(x)\times(t-r^2,t) \quad \mbox{with}\quad z=(x,t).$$

We write $\partial_j u_i = \dfrac{\partial u_i}{\partial x_j}$ and use the letters $C$ or $c$ to denote generic constants that could be different from line to line.

%For $a<b$ and a measurable set$A\subset\R^3$, the space  $L^s(a,b; L^p(A))$ is the set of all measurable functions $u(x,t)$ in $A \times (a,b)$
%such that $\norm{u}_{L^s(a,b; L^p(A))}$ is finite. Here
%$$\norm{u}_{L^s(a,b; L^p(A))}:=\left\{ \begin{array}{lcl}
% \ds \left(\int_{a}^{b}\norm{u(\cdot,t)}^s_{L^p(A)} dt\right)^{\frac{1}{s}},   &\qquad& {\rm if~} s\in[1,\infty), \\
%\, \   \ds \esssup_{t\in(a,b)}\norm{u(\cdot,t)}_{L^p(A)}, &\qquad& {\rm if~} s=\infty.
%\end{array}\right.
%$$

For each  bounded open set $O\subset\R^3$, we denote by $L^{-1,2}(O)$ the dual of the Sobolev space $W^{1,2}_0(O)$. The latter is defined as the completion of 
$C_0^\infty(O)$ (the space of smooth functions with compact support in $O$) under the Dirichlet integral
$$\norm{\varphi}_{W^{1,2}_0(O)}=\left(\int_{O}|\nabla \varphi|^2 dx\right)^{\frac{1}{2}}.$$ 

We shall use the following well-known representation of a function $\varphi\in C_0^\infty(\R^3)$
\begin{equation}\label{rep}
\varphi(x)=\frac{1}{|S^2|}\int_{\R^3} \frac{(x-y)\cdot \nabla\varphi(y)}{|x-y|^3} dy,
\end{equation}
where $|S^2|$ is the area of the unit sphere (see, e.g., \cite[p. 125]{Stein70}). Identity \eqref{rep} can be used to show that  
\begin{eqnarray} \label{dualest}
\norm{f}_{L^{-1,2}(O)} \leq C  \norm{ {\rm\bf I}_1(\chi_{O}|f|)}_{L^2(O)}
\end{eqnarray}
for any $f\in L^1_{\rm loc}(\R^3)$ and any bounded open set $O\subset\R^3$.
Indeed,  for any $\varphi\in C_0^{\infty}(O)$ by \eqref{rep} it holds that
\begin{eqnarray*}
\left |\int_{O} \varphi(x) f(x)dx \right | &\leq& C \int_{O} \left[\int_{O}\frac{|\nabla \varphi(y)|}{|x-y|^{2}} dy\right]
|f(x)| dx\\
&=& C  \int_{O}|\nabla \varphi(y)| \left[\int_{O}\frac{ |f(x)|dx}{|x-y|^{2}}\right]
dy\\
&\leq& C \norm{\nabla \varphi}_{L^2(O)} \norm{ {\bf I}_1(\chi_{O}|f|)}_{L^2(O))},
\end{eqnarray*}
as desired.

The following lemma will be needed later. Its proof is based on a simple iteration and can be found in \cite[Lemma 6.1]{Giu03}.

\begin{lemma}\label{Giusti}
Let $I(s)$ be a bounded nonnegative function in the interval $[R_1, R_2]$. Assume that for every $s, \rho\in [R_1, R_2]$ and  $s<\rho$ we have 
$$I(s)\leq [A(\rho-s)^{-\alpha} +B(\rho-s)^{-\beta} +C] +\theta I(\rho)$$
with  $A, B, C\geq 0$, $\alpha>\beta>0$ and $\theta\in [0,1)$. Then it holds that
$$I(R_1)\leq c(\alpha, \theta) [A(R_2-R_1)^{-\alpha} +B(R_2-R_1)^{-\beta} +C].$$
\end{lemma}

We now make precise the definition of a weak solution to the system \eqref{SNSE}.
\begin{definition}\label{defweak}
A divergence-free vector field $U = (U_1,U_2,U_3)$  is called a weak solution of \eqref{SNSE} if $U\in W^{1,2}_{\rm loc}(\R^3)$ and if for all divergence-free vector field $\phi = (\phi_1,\phi_2,\phi_3) \in C^\infty_0(\R^3)$ one has 
\begin{equation}\label{weakform}
\int_{\R^3}(\nu \nabla U\cdot\nabla \phi +[aU+a(y\cdot \nabla)U+(U\cdot\nabla) U]\cdot \phi) dy   = 0.
\end{equation}
\end{definition}

\section{The pressure formulation}
It is known that every weak solution $U$ of \eqref{SNSE} is smooth (see \cite{Ga94, GiMo82,Lady69,Temam77}). Note that Definition \ref{defweak} does not include a pressure $P$. However, taking the divergence of \eqref{SNSE} we formally obtain the pressure equation\footnote{Here and in what follows we use the usual convention and sum over the repeated indices.}
\begin{equation}\label{pres}
-\Delta P= \partial_i\partial_j(U_iU_j).
\end{equation}

The main goal of this section is to recover a $P$ (with a useful control) from a weak solution $U$ of \eqref{SNSE} for which \eqref{local-cond-on-U} holds. As $P$ is  generally a signed distribution, an estimate of the form \eqref{local-cond-on-U} should not hold if $|U|^2$ is replaced by $|P|$. On the other hand, we observe that \eqref{local-cond-on-U} is equivalent to the condition 
$$\norm{|U|^2}_{L^{-1.2}(B_r)}^2\leq C \, r^{3-\gamma} \qquad \forall {\rm ~balls~} B_r\subset\R^3,$$
for some $\gamma\in(0,3]$. Thus it is natural to expect that the pressure $P$ should also satisfy a similar condition in which $|U|^2$ is replaced by $P$.

To construct such a $P$ we start with the following lemma.
\begin{lemma}\label{GIJ} Let $U=(U_1, U_2, U_3)$ satisfy \eqref{local-cond-on-U} for some $\gamma\in(0,3]$. 
For each  $i,j\in\{1, 2, 3\}$  there exists a vector field $G_{ij}\in \M^{2,\gamma}(\R^3)$ such that 
\begin{equation}\label{GIJUIJ}
-\, {\rm div} \, G_{ij}=U_i U_j
\end{equation}
in the sense of distributions, i.e., 
$$\int_{\R^3} G_{ij} \cdot \nabla \varphi dx = \int_{\R^3} U_i U_j \varphi dx$$
for all $\varphi\in C_0^\infty(\R^3)$.
\end{lemma}

\begin{proof}

By hypothesis, we see that for every $i,j\in\{1,2,3\}$,
$$\int_{B_r} {\bf I}_1(\chi_{B_r}|U_iU_j|)^2 dx\leq C\, r^{3-\gamma},\qquad \forall B_r\subset\R^3.$$

Let $\{\Psi_N(x)\}$ be a sequence of smooth functions in $\R^3$ such that $0\leq \Psi_N\leq 1$, $\Psi_N(x)=1$ for $|x|\leq N/2$, $\Psi_N(x)=0$ for 
$|x|\geq N$, and $|\nabla \Psi_N(x)| \leq c/N$.  Thus we also have 
\begin{equation*}
\int_{B_r} {\bf I}_1(\chi_{B_r} \Psi_{N}|U_iU_j|)^2 dx\leq C\, r^{3-\gamma},\qquad \forall B_r\subset\R^3,
\end{equation*}
which by \eqref{dualest} yields
\begin{equation}\label{trunc}
\norm{\Psi_NU_iU_j}_{L^{-1,2}(B_r)}\leq C\, r^\frac{3-\gamma}{2},\qquad \forall B_r\subset\R^3,
\end{equation}
where $C$ is independent of $N$.

For each $N\geq 1$, we claim that there exists a vector field $G_{ij}^{N}\in \M^{2,\gamma}(\R^3)$ such that 
\begin{equation}\label{FandU}
-\, {\rm div} \, G^N_{ij}=\Psi_NU_i U_j
\end{equation}
and 
\begin{equation}\label{uni-M-bound111}
G^N_{ij}\rightarrow G_{ij} {\rm ~strongly~ in~} \M^{2,\gamma}(\R^3)
\end{equation} 
as $N\rightarrow +\infty$ for some vector field $G_{ij}$. In particular, 
$\lim_{N\rightarrow\infty}\langle G^{N}_{ij}, \phi \rangle =\langle G_{ij}, \phi \rangle$
for every vector field $\phi\in C_0^\infty(\R^3)$. In view of  \eqref{FandU}, this gives   
\eqref{GIJUIJ} as  as desired. 

Thus it is left to show \eqref{FandU} and \eqref{uni-M-bound111}. To that end, we define 
$$G^{N}_{ij}(x)=\frac{1}{|S^2|}\int_{\R^3} \frac{y-x}{|y-x|^3}\Psi_N(y) U_i(y)U_j(y) dy.$$
Note that  $G^N_{i,j}=\nabla(-\Delta)^{-1}(\Psi_N U_iU_j)$ in the sense of distributions, i.e.,
$$\langle G^N_{ij}, \phi\rangle :=\langle \Psi_N U_iU_j, -{\rm div}\, {\bf I}_2(\phi)\rangle= \langle  U_iU_j, -\Psi_N \, {\rm div}\, {\bf I}_2(\phi)\rangle $$
for  $\phi=(\phi_1, \phi_2, \phi_3)\in C_0^\infty(\R^3)$.  Then obviously, \eqref{FandU} holds in the sense of distributions.

On the other hand, for every $k\in\{1,2,3\}$, $N,M \geq 1$, and every scalar function $\varphi\in C_0^\infty(B_r)$ we have 
\begin{equation}\label{part-sum}
\langle (\Psi_N-\Psi_M) U_iU_j, \partial_k  {\bf I}_2(\varphi)\rangle= \sum_{\nu=0}^\infty \langle (\Psi_N-\Psi_M) U_iU_j, \eta_\nu\partial_k  {\bf I}_2(\varphi)\rangle,
\end{equation}
where $\{\eta_\nu\}_{\nu=0}^{\infty}$ is a smooth partition of unity associated to the ball $B_r$. That is, $\eta_0\in C_0^\infty(B_{2r})$, $\eta_{\nu}\in C_0^\infty(B_{2^{\nu+1}r}\setminus B_{2^{\nu-1}r})$, $\nu=1,2,3,\dots$, such that  
$$0\leq \eta_{\nu}\leq 1, \qquad |\nabla\eta_{\nu}|\leq c (2^j r)^{-1}, \qquad \nu=0, 1, 2, \dots, $$
and 
$$\sum_{\nu=0}^\infty \eta_{\nu}(x)=1,  \qquad x\in\R^3.$$

Note that the sum in \eqref{part-sum} has only a finite number of non-zero terms. Moreover, by  the property of $\{\Psi_N\}$ we have 
\begin{equation*}
\langle (\Psi_N-\Psi_M) U_iU_j, \partial_k  {\bf I}_2(\varphi)\rangle= \sum_{\nu=N_0}^\infty \langle (\Psi_N-\Psi_M) U_iU_j, \eta_\nu\partial_k  {\bf I}_2(\varphi)\rangle,
\end{equation*}
where $N_0\rightarrow +\infty$ as $M, N\rightarrow +\infty$. Using this and \eqref{trunc}, we have 
\begin{eqnarray}\label{IV}
\lefteqn{\left|\langle (\Psi_N-\Psi_M)  U_iU_j, \partial_k  {\bf I}_2(\varphi)\rangle\right|}\notag\\
&\leq& c \sum_{\nu=N_{0}}^\infty (2^{\nu} r)^{\frac{3-\gamma}{2}}\norm{\nabla[\eta_\nu\partial_k  {\bf I}_2(\varphi)]}_{L^2(B_{2^{\nu+1}r})}\notag\\
 &\leq& c \sum_{\nu=N_0}^\infty (2^{\nu} r)^{\frac{3-\gamma}{2}} 2^{-\nu\frac{3}{2}}\norm{\varphi}_{L^2(B_{r})}\\
&\leq& c\, r^{\frac{3-\gamma}{2}} \norm{\varphi}_{L^2(B_{r})}\sum_{\nu=N_0}^\infty 2^{-\frac{\nu \gamma}{2}}.\notag
\end{eqnarray}
Here in \eqref{IV} we used the bound
$$\norm{\nabla[\eta_\nu\partial_k {\bf I}_2(\varphi)]}_{L^2(B_{2^{\nu+1}r})}\leq c\, 2^{-\nu\frac{3}{2}}\norm{\varphi}_{L^2(B_{r})},$$
which holds for all $\nu=0, 1, 2, \dots$, and $\varphi\in C_0^\infty(B_r)$ (see \cite[Proposition 4.2(ii)]{MV06}).

Therefore, for  all $\phi=(\phi_1, \phi_2, \phi_3)\in C_0^\infty(B_r)$, we find
\begin{eqnarray*}
\left|\langle  G^N_{ij}-G^M_{ij}, \phi\rangle\right| &=&\left|\langle (\Psi_N-\Psi_M) U_iU_j, - \, {\rm div}\, {\bf I}_2(\phi)\rangle\right|\\
&\leq& c\, r^{\frac{3-\gamma}{2}} \norm{\phi}_{L^2(B_{r})}\sum_{\nu=N_0}^\infty 2^{-\frac{\nu \gamma}{2}}.
\end{eqnarray*}
That is, by duality 
\begin{equation*}
\norm{G^N_{ij}-G^M_{ij}}_{\M^{2,\gamma}(\R^3)} \leq c\, \sum_{\nu=N_0}^\infty 2^{-\frac{\nu \gamma}{2}}\rightarrow 0 {\rm ~as~} M,N \rightarrow+\infty.
\end{equation*}

A similar (and simpler) argument also shows that $G_{i,j}^N\in \mathcal{M}^{2,\gamma}(\R^3)$ for all $N\geq1$. Thus $\{G^N_{ij}\}$ is a Cauchy sequence in   $\M^{2,\gamma}(\R^3)$ which  converges to a limit $G_{i,j}$
as claimed in \eqref{uni-M-bound111}.
\end{proof}

We are now ready to construct the desired pressure $P$.

\begin{lemma} \label{lem:pressure}
Suppose that  $U$ is a weak solution of \eqref{SNSE}  such that \eqref{local-cond-on-U} holds  for some $0<\gamma \leq3$. Let $R_j=\partial_j(-\Delta)^{\frac{1}{2}}$ be the $j$-th Riesz transform for $j=1,2,3$. Define a distribution  $P$ by letting 
$$\ds \langle P, \varphi\rangle= \langle R_i R_j(G_{ij}) , \nabla \varphi\rangle, \qquad \varphi\in C_0^\infty(\R^3),$$
where $G_{ij}=(G_{ij}^{1}, G_{ij}^{2}, G_{ij}^3)\in \mathcal{M}^{2, \gamma}(\R^3)$ is given by Lemma \ref{GIJ}, 
and $$R_i R_j(G_{ij})=(R_i R_j(G^1_{ij}), R_i R_j(G^2_{ij}), R_i R_j(G^3_{ij})).$$  Then $P$ satisfies the following growth estimate  
\begin{equation}\label{Pgrowth}
 \norm{P}_{ \Lp{-1,2}{B_r(x)}}
 \leq C\, r^{\frac{3-\gamma}2}, \qquad \forall x\in\R^3, r>0.
\end{equation}
Moreover, $P$   satisfies \eqref{pres} and  $(U,P)$ smoothly solves 
\begin{equation}\label{prescon}
-\nu \Delta U+aU+a(y\cdot\nabla)U+(U\cdot\nabla)U+\nabla P=0 \quad {\rm in} \, \R^3.
\end{equation}
\end{lemma}

\begin{proof}
By Lemma \ref{GIJ}, the vector fields $G_{ij}\in \M^{2,\gamma}(\R^3)$. Since $R_iR_j$ is bounded on $\M^{2,\gamma}(\R^3)$ (see, e.g., \cite{Peetre66}) this implies that 
$P$ is well-defined and \eqref{Pgrowth} holds.   Indeed, for any $\varphi\in C_0^\infty(B_r(x))$ we have
\begin{eqnarray*}
|\langle P, \varphi\rangle| &=&\left|\int_{B_r(x)} (R_i R_j(G_{ij}))\cdot\nabla\varphi dy\right|\\
&\leq& \norm{R_i R_j(G_{ij})}_{L^2(B_r(x))} \norm{\nabla \varphi}_{L^2(B_r(x))}\\
&\leq& C \norm{G_{ij}}_{\M^{2,\gamma}(\R^3)} r^{\frac{3-\gamma}2} \norm{\nabla \varphi}_{L^2(B_r(x))},
\end{eqnarray*}
which obviously yields \eqref{Pgrowth}.

Using  the facts that $-{\rm div} G_{ij}= U_iU_j$  and $R_iR_j\Delta \varphi=-\partial_i\partial_j\varphi$ for any $\varphi\in C_0^\infty(\R^3)$ (see \cite[p. 59]{Stein70})  we can now calculate
\begin{eqnarray}\label{explain}
\langle\Delta P,\varphi\rangle&=& \int_{\R^3}(R_iR_j(G_{ij}))\cdot\nabla \Delta\varphi dx\notag\\
&=&\int_{\R^3} G_{ij} \cdot \nabla R_iR_j(\Delta \varphi) dx \\
&=&\int_{\R^3} U_iU_j  R_iR_j(\Delta \varphi) dx\notag\\
&=&- \int_{\R^3} U_iU_j  \partial_i\partial_j \varphi dx. \notag
\end{eqnarray}

That is, $P$ is a distributional solution of \eqref{pres} and thus by Weyl's lemma it is smooth (since $U_iU_j$ is smooth). Note that the second equality in \eqref{explain} requires an explanation as in general $|G_{ij}|\not\in L^2(\R^3)$ unless $\gamma=3$. But since $|G_{ij}| \in \M^{2,\gamma}(\R^3)$ we have $|G_{ij}|\in L^2_{\rm loc}(\R^3)$ and moreover
\begin{equation}\label{weightM}
\int_{\R^3} |G_{ij}|^2 (1+|x|^2)^{-\frac{\mu}{2}} dx \leq c\, \norm{G_{ij}}_{\M^{2,\gamma}}^{2} <+\infty. 
\end{equation}
for any $\mu>3-\gamma$. Inequality \eqref{weightM}  can be found in \cite{Kato92}, page 132.

Now let $\chi_{B_R(0)}$ be the characteristic function of $B_R(0)$. Using H\"older's inequality we have 
\begin{eqnarray*}
\lefteqn{\left|\int_{\R^3}(R_iR_j(G_{ij})-R_iR_j( \chi_{B_R(0)} G_{ij})) \cdot\nabla \Delta\varphi dx\right|}\\
&\leq&  \left(\int_{\R^3}\left|R_iR_j(G_{ij}- \chi_{B_R(0)} G_{ij})\right|^2 w(x) dx\right)^{\frac12} \left(\int_{\R^3} |\nabla \Delta\varphi|^2 w(x)^{-1} dx\right)^{\frac12},
\end{eqnarray*}
where we choose the weight $w(x)=(1+|x|^2)^{-\frac{\mu}{2}}$. Note that $w$ belongs to the Muckenhoupt $A_2$ class provided we choose $\mu\in(3-\gamma,3)$ (see, e.g., \cite[Chap. V]{Stein93}). Since $R_iR_j$ is bounded on 
the weighted space $L^2_{w}(\R^3)$ (see \cite[p. 205]{Stein93}) we then have  
\begin{eqnarray*}
\lefteqn{\left|\int_{\R^3}(R_iR_j(G_{ij})-R_iR_j( \chi_{B_R(0)} G_{ij})) \cdot\nabla \Delta\varphi dx\right|}\\
&\leq& C \, \left(\int_{\R^3}\left| G_{ij}- \chi_{B_R(0)} G_{ij}\right|^2 w(x) dx\right)^{\frac12}. 
\end{eqnarray*}

Thus using \eqref{weightM} and the Lebesgue's dominated convergence theorem, we eventually find
\begin{eqnarray*}
\int_{\R^3}(R_iR_j(G_{ij})) \cdot\nabla \Delta\varphi dx&=&\lim_{R\rightarrow +\infty} \int_{\R^3}(R_iR_j( \chi_{B_R(0)} G_{ij}))\cdot\nabla \Delta\varphi dx\\
&=& \lim_{R\rightarrow +\infty} \int_{\R^3}( \chi_{B_R(0)} G_{ij})\cdot  R_iR_j(\nabla \Delta\varphi) dx\\
&=& \int_{\R^3}G_{ij}\cdot\nabla R_iR_j(\Delta\varphi) dx,
\end{eqnarray*}
as desired. Here the second equality follows since $\chi_{B_R(0)} G_{ij}\in L^2(\R^3)$, and the last equality follows 
since $R_iR_j(\nabla\Delta\varphi)=\nabla R_iR_j(\Delta\varphi)$ has compact support.

Finally, to prove \eqref{prescon}, we let 
\begin{equation*}%\label{defF}
F:= -\nu \Delta U+aU+a(y\cdot\nabla)U+(U\cdot\nabla)U+\nabla P
\end{equation*}
 and show that $F\equiv 0.$

Using \eqref{weakform} with appropriate test functions $\phi$ we find $\curl F=0$. Also, by \eqref{pres} and the fact that $\dive U=0$ we have $\dive F=0$. These imply that $\Delta F=0$. Thus by the mean-value property of harmonic functions we find
\begin{equation}\label{MVP}
F(0)=\epsilon^3\int_{\R^3} F(y)\varphi(\epsilon y)dy
\end{equation}
for any $\epsilon>0$ and every radial function $\varphi\in C_0^\infty(B_{1}(0))$ such that $\ds\int_{\R^3}\varphi d y=1$ (see \cite[p. 275]{Stein70}).

Following \cite{NeRuSv96}, by analyticity, to show that $F\equiv 0$  it is enough to verify that $D^\alpha F(0)=0$ for each 
for  each multi-index $\alpha=(\alpha_1,\alpha_2,\alpha_3)$ with $|\alpha|\geq 0$. To this end, we first apply \eqref{MVP} to the harmonic function 
$D^\alpha F$ and then integrate by parts to obtain
\begin{equation*}%\label{MVPD}
D^\alpha F(0)=\epsilon^3\int_{\R^3} D^\alpha F(y)\varphi(\epsilon y)dy =(-1)^{|\alpha|}\epsilon^3\int_{\R^3}  F(y) \epsilon^{|\alpha|} (D^\alpha\varphi)(\epsilon y)dy,
\end{equation*}
where $\epsilon$ and $\varphi$ are as above. Thus to show that $D^\alpha F(0)=0$, it is enough to verify that
$$\lim_{\epsilon\to 0^+}\epsilon^3\int_{\R^3}F(y)\theta(\epsilon y)dy=0$$ 
for any  function $\theta\in C_0^\infty(B_1(0))$. That is, we need to show that
\begin{equation}\label{5term}
\lim_{\epsilon\to 0^+}\epsilon^3  \int_{\R^3}[-\nu \Delta U+aU+a(y\cdot\nabla)U+(U\cdot\nabla)U+\nabla P]\theta(\epsilon y)dy=0.
\end{equation}
The first four terms in the above expression can be treated similarly. For example, for the term involving $a(y\cdot\nabla)U$ by integrating 
by parts we have 
\begin{eqnarray*}
\lefteqn{\epsilon^3\int(y\cdot\nabla) U(y)\theta(\epsilon y) dy}\\
&=&-\epsilon^3\int 3U(y)\theta(\epsilon y)dy- \epsilon^3 \int U(y)(\epsilon y_i)(\partial_i\theta)(\epsilon y)dy\\
&=& -3\epsilon^3 \int U(y)\theta(\epsilon y)dy -\epsilon^3\int U(y)\widetilde{\theta}(\epsilon y)dy,
\end{eqnarray*}
where  $\widetilde{\theta}(y)=(\epsilon y_i)(\partial_i\varphi)(y)$.

On the other hand, by H\"older's inequality and a simple change of variables we find  
\begin{eqnarray*}
\left| \int_{\R^3} U(y)\theta(\epsilon y)dy\right|
&\leq& \left(\int_{B_{1/\epsilon}} |U|^2(y) |\theta(\epsilon y)|dy\right)^{\frac12} \left(\int_{B_{1/\epsilon}} |\theta(\epsilon y)|dy\right)^{\frac12}\\
&\leq&\epsilon^{-\frac32} \left(\int_{B_{1/\epsilon}} |U|^2(y)|\theta(\epsilon y)|dy\right)^{\frac12} \left(\int_{B_{1}} |\theta( z)|dz\right)^{\frac12}\\
&\leq& c\, \epsilon^{-\frac32} \left(\int_{B_{1/\epsilon}} |U|^2(y)|\theta(\epsilon y)|dy\right)^{\frac12}.
\end{eqnarray*}

Thus by \eqref{dualest} it holds that 
\begin{eqnarray*}
\lefteqn{\left| \int_{\R^3} U(y)\theta(\epsilon y)dy\right|}\\
&\leq& c\, \epsilon^{-\frac32}\left(\norm{|U|^2}_{L^{-1,2}(B_{1/\epsilon})}\right)^{\frac12}  \left(\int_{B_{1/\epsilon}}\left|\epsilon (\nabla |\theta|)(\epsilon y)\right|^2dy\right)^{\frac14}\\
&\leq& c\, \epsilon^{-\frac32+\frac12} \left(\int_{B_{1/\epsilon}} \left|\Ri{1}\left(|U|^2\right)\right|^2dy\right)^{\frac14}\left(\int_{B_{1/\epsilon}}\left| (\nabla \theta)(\epsilon y)\right|^2dy\right)^{\frac14}\\
&\leq& c\, \epsilon^{-1-\frac34} \left(\int_{B_{1/\epsilon}} \left|\Ri{1}\left(|U|^2\right)\right|^2dy\right)^{\frac14}\left(\int_{B_{1}}\left| \nabla \theta(z)\right|^2dz\right)^{\frac14},
\end{eqnarray*} 
which by our assumption on $U$ yields
\begin{equation*}
\left| \int_{\R^3} U(y)\theta(\epsilon y)dy\right| \leq c\, \epsilon^{-\frac74} \left(1/\epsilon\right)^{\frac{3-\gamma}4}= c\, \epsilon^{-\frac 52+\frac{\gamma}4}.
\end{equation*} 

Of course, the same inequality also holds with $\widetilde{\theta}$ in place of $\theta$. Hence, using these in the expression  for
$\epsilon^3\int(y\cdot\nabla) U(y)\theta(\epsilon y) dy$ above we obtain
\begin{equation*}
\left|\epsilon^3\int(y\cdot\nabla) U(y)\theta(\epsilon y) dy\right| \leq c\, \epsilon^{\frac 12+\frac{\gamma}4}\rightarrow 0 {\rm~ as ~ } \epsilon\rightarrow 0.
\end{equation*}

For what concerns the term involving $\nabla P$ in \eqref{5term}, using integration by parts and the bound \eqref{Pgrowth} we have 
 \begin{eqnarray*}
\left|\epsilon^3\int \nabla P\theta(\epsilon y)dy\right|&=& \left|-\epsilon^4 \int P(y) (\nabla\theta)(\epsilon y)dy\right|\\
&\leq& c\, \epsilon^{4} \norm{P}_{L^{-1,2}(B_{1/\epsilon})} \left(\int_{B_{1/\epsilon}}\left|\epsilon (\nabla^2 \theta)(\epsilon y)\right|^2dy\right)^{\frac12}\\
&\leq& c\, \epsilon^{5} (1/\epsilon)^{\frac{3-\gamma}{2}} \left(\int_{B_{1/\epsilon}}\left| (\nabla^2 \theta)(\epsilon y)\right|^2dy\right)^{\frac12}\\
&\leq& c\, \epsilon^{5-\frac32-\frac32+\frac{\gamma}{2}}\left(\int_{B_{1}}\left| \nabla^2 \theta(z)\right|^2dz\right)^{\frac12}.
\end{eqnarray*} 

Thus we also have 
$$\left|\epsilon^3\int \nabla P\theta(\epsilon y)dy\right| \rightarrow 0 {\rm~ as ~ } \epsilon\rightarrow 0.$$

In conclusion, we obtain \eqref{5term} and that completes the proof of the lemma.
\end{proof}

\section{Proof of Theorem \ref{Morrey-loc}}\label{sec4}

We are now ready to prove Theorem \ref{Morrey-loc}.
\begin{proof}[Proof of Theorem \ref{Morrey-loc}]
Obviously, it is enough to prove the first statement of the theorem. Henceforth, suppose that $U$ is  a weak solution of \eqref{SNSE} such that condition \eqref{local-cond-on-U} holds
 for some $\gamma\in(0,3]$.
Let the function $P$ be defined as in Lemma \ref{lem:pressure}. Then $(U, P)$ smoothly solves  \eqref{SNSE}, and thus the functions
$$u(x,t)=\lambda(t)U(\lambda(t)x), \quad {\rm and} \quad  p(x,t)=\lambda^2(t)P(\lambda(t)x),$$
with $\lambda(t)=[2a(T-t)]^{-1/2}$, $T>0$,  solves the Navier--Stokes equations \eqref{NSE} in the classical sense in 
$\R^3\times (-\infty,T)$. By Theorem \ref{T1.1}, it is enough to check that 
\begin{equation}\label{needtocheck}
\esssup_{T-1<t<T} \int_{B_{1}(0)}\dfrac12|u(x, t)|^2dx+\int_{T-1}^{T}\int_{B_{1}(0)}\nu|\nabla u(x,t)|^2dxdt<+\infty.
\end{equation}

To this end, we first observe that for any ball $B_r(x_0)\subset\R^3$ and $t<T$, it holds that 
\begin{equation}\label{upbound}
 \norm{|u(\cdot, t)|^2}_{ \Lp{-1,2}{B_r(x_0)}} + \norm{p(\cdot, t)}_{ \Lp{-1,2}{B_r(x_0)}}\leq c\, \lambda(t)^{1-\frac{\gamma}{2}} r^{\frac{3-\gamma}2}.
\end{equation}

Indeed for any $\varphi\in C^{\infty}_0(B_r(x_0))$ and with $\lambda=\lambda(t)$ we have 
\begin{eqnarray*}
\left|\int_{B_r(x_0)}|u(x,t)|^2\varphi(x)dx\right|&=&\lambda^2\left|\int_{B_r(x_0)}| U(\lambda x)|^2\varphi(x)dx\right|\\
&=&\lambda^{-1} \left|\int_{B_{r\lambda}(\lambda x_0)}|U(z)|^2\varphi\left(z/\lambda \right)dz\right|.
\end{eqnarray*}

From this using inequality \eqref{dualest} we obtain

\begin{eqnarray*}
\lefteqn{\left|\int_{B_r(x_0)}|u(x,t)|^2\varphi(x)dx\right|}\\
&\leq& c\,\lambda^{-1} \norm{|U|^2}_{L^{-1,2}({B_{r\lambda}(\lambda x_0)})}\left(\int_{B_{r\lambda}(\lambda x_0)} \left| \frac1\lambda \nabla \varphi\left(\frac z \lambda \right)\right|^2dz\right)^\frac12\\
&\leq& c\, \lambda^{-2} (r\lambda)^{\frac{3-\gamma}2}\left(\int_{B_{r}(x_0)} \left|\nabla \varphi\left(x\right)\right|^2{\lambda^3}{dx}\right)^\frac12\\
&\leq&\lambda^{1-\frac{\gamma}{2}} r^{\frac{3-\gamma}2}\left(\int_{B_{r}(x_0)} \left|\nabla \varphi\left(x\right)\right|^2{dx}\right)^\frac12.
\end{eqnarray*}

This gives 
$$\norm{|u(\cdot, t)|^2}_{ \Lp{-1,2}{B_r(x_0)}} \leq c\, \lambda(t)^{1-\frac{\gamma}{2}} r^{\frac{3-\gamma}2}.$$ 

Likewise, using the bound \eqref{Pgrowth} and an analogous  argument we obtain a similar bound for $p$. Thus  \eqref{upbound} is proved.

Next we shall make use of the well-known energy equality:
\begin{eqnarray*}
\lefteqn{ \int_{B_R}|u(x, t)|^2 \phi(x,t) dx + 2\nu\int_{T_1}^t\int_{B_R} |\nabla u|^2 \phi(x, s) dxds} \\
&=& \int_{T_1}^t\int_{B_R} |u|^2 (\phi_t +\nu\Delta \phi) dx ds + \int_{T_1}^t\int_{B_R}(|u|^2 + 2p)u\cdot \nabla \phi dx ds,
\end{eqnarray*}
which holds for every ball $B_R=B_R(0)$, $t\in (T_1, T)$, and any nonnegative function $\phi \in C_0^{\infty}(\R^3\times\R)$ vanishing in a neighborhood of the parabolic boundary $B_R\times\{t=T_1\} \cup \partial B_R\times[T_1, T]$ of the cylinder $B_R\times(T_1,T)$.

Let  $T_\epsilon=T-\epsilon$ for sufficiently small $\epsilon>0$, say, $\epsilon\in(1/2,0)$. For any balls
$$B_s=B_s(0),\quad  B_\rho=B_\rho(0),$$
with $1\leq s<\rho\leq 2$, we consider a test function  $\phi(x,t)=\eta_1(x)\eta_2(t)$ where $\eta_1\in  C_0^{\infty}(B_\rho)$, $0\leq \eta_1\leq 1$ in $\R^n$, $\eta_1\equiv 1$ on $B_s$, and $$|\nabla^{\alpha} \eta_1|\leq \frac{c}{(\rho-s)^{|\alpha|}}$$ for all multi-indices $\alpha$ with $|\alpha|\leq 3$. The function $\eta_2(t)$ is chosen so that 
$\eta_2\in C_0^{\infty}(T_\epsilon-\rho^2, T_{\epsilon}+\rho^2)$, $0\leq\eta_2\leq 1$ in $\R$, $\eta_2(t)\equiv 1$ for $t\in [T_\epsilon-s^2, T_\epsilon+s^2]$, and $$|\eta'_2(t)|\leq \frac{c}{\rho^2-s^2}\leq \frac{c}{\rho-s}.$$

Thus
$$|\nabla \phi_t|+ |\nabla \Delta \phi|\leq  \frac{c}{(\rho-s)^3}, \quad | \nabla^2 \phi|\leq  \frac{c}{(\rho-s)^2},\quad |\nabla  \phi|\leq  \frac{c}{\rho-s}.$$

Let 
$$I(s,\epsilon)=\sup_{T_\epsilon-s^2\leq t\leq T_\epsilon}\int_{B_s}|u(x, t)|^2 dx + \int_{T_\epsilon-s^2}^{T_\epsilon}\int_{B_s}|\nabla u(x, t)|^2 dx dt,$$
which is a finite quantity provided $\epsilon>0$. 

Using $\phi$ as a test function in the  energy equality above we have
\begin{eqnarray*}
I(s,\epsilon)&\leq& c\,\int_{T_\epsilon-\rho^2}^{T_\epsilon} \norm{|u|^2}_{L^{-1, \, 2}(B_\rho)}\norm{\nabla \phi_t +\nabla \Delta \phi}_{L^2(B_{\rho})} dt+ \nonumber\\
\quad&& +\, c\,\int_{T_\epsilon-\rho^2}^{T_\epsilon} \norm{|u|^2 + 2p}_{L^{-1, \, 2}(B_\rho)}\norm{\nabla u\cdot \nabla \phi + u\cdot \nabla^2\phi}_{L^2(B_{\rho})} dt\nonumber \\
&\leq& c\, \int_{T_\epsilon-\rho^2}^{T_\epsilon} \lambda(t)^{1-\frac{\gamma}{2}}\norm{\nabla \phi_t +\nabla \Delta \phi}_{L^2(B_{\rho})} dt+ \nonumber\\
\quad&& +\, c\, \int_{T_\epsilon-\rho^2}^{T_\epsilon} \lambda(t)^{1-\frac{\gamma}{2}}\norm{\nabla u\cdot \nabla \phi + u\cdot \nabla^2\phi}_{L^2(B_{\rho})} dt,
\end{eqnarray*}
where we  used \eqref{upbound} and the fact that $1\leq \rho\leq 2$. From the choice of $\phi$, this gives 
\begin{eqnarray*}
I(s,\epsilon)&\leq&  \frac{c}{(\rho-s)^3}\, \int_{T_\epsilon-\rho^2}^{T_\epsilon} \lambda(t)^{1-\frac{\gamma}{2}}  dt+ \nonumber\\
\quad&& +\, c\, \int_{T_\epsilon-\rho^2}^{T_\epsilon} \lambda(t)^{1-\frac{\gamma}{2}}\left(\frac{\norm{\nabla u}_{L^2(B_\rho)}}{\rho-s} +
\frac{\norm{u}_{L^2(B_\rho)}}{(\rho-s)^2}\right) dt.
\end{eqnarray*}

Thus by H\"older's inequality we find
\begin{eqnarray*}
I(s,\epsilon)&\leq&  \frac{c}{(\rho-s)^3}\, \int_{T_\epsilon-\rho^2}^{T_\epsilon} \lambda(t)^{1-\frac{\gamma}{2}}  dt+ \nonumber\\
&& +\frac{c}{\rho-s}\, \left(\int_{T_\epsilon-\rho^2}^{T_\epsilon} \lambda(t)^{2-\gamma} dt \right)^{\frac12}
\left(\int_{T_\epsilon-\rho^2}^{T_\epsilon} \norm{\nabla u}_{L^2(B_\rho)}^2 dt \right)^{\frac12}+\\ 
&&+ \frac{c}{(\rho-s)^2} \left(\int_{T_\epsilon-\rho^2}^{T_\epsilon} \lambda(t)^{2-\gamma} dt \right)^{\frac12}
\left(\sup_{T_\epsilon-\rho^2\leq t\leq T_\epsilon} \norm{u(\cdot, t)}_{L^2(B_\rho)}^2 \right)^{\frac12}.
\end{eqnarray*}

Now using Young's inequality we arrive at 
\begin{eqnarray*}
I(s,\epsilon)&\leq&  \frac{c}{(\rho-s)^3}\, \int_{T_\epsilon-\rho^2}^{T_\epsilon} \lambda(t)^{1-\frac{\gamma}{2}}  dt+ \nonumber\\
\qquad\qquad&& +\left(\frac{c}{(\rho-s)^2} + \frac{c}{(\rho-s)^4} \right) \int_{T_\epsilon-\rho^2}^{T_\epsilon} \lambda(t)^{2-\gamma} dt + \frac{1}{2} I(\rho,\epsilon)
\end{eqnarray*}

As this holds for all $1\leq s<\rho\leq 2$ by Lemma \ref{Giusti} we find 
\begin{eqnarray*}
I(1,\epsilon) &\leq&  C\, \int_{T_\epsilon-4}^{T_\epsilon} \lambda(t)^{1-\frac{\gamma}{2}}  dt+ C\, \int_{T_\epsilon-4}^{T_\epsilon} \lambda(t)^{2-\gamma}  dt\\
&\leq&  C\, \int_{T-5}^{T} [\lambda(t)^{1-\frac{\gamma}{2}}   + \lambda(t)^{2-\gamma}]  dt\\
&\leq&  C(a, \gamma) <+\infty.
\end{eqnarray*}

Since this holds for every  $\epsilon\in(1/2,0)$ we deduce that \eqref{needtocheck} holds and thus the proof is complete.
\end{proof}

\bibliographystyle{plain}

\end{document}